\documentclass[10pt]{amsart}
\title{Geometry of representations of quantum planes}
\author{Kevin De Laet}
\address{Mathematics and statistics, Hasselt University \\
Agoralaan - Building D, B-3590 Diepenbeek (Belgium) \\ {\tt kevin.delaet@uhasselt.be}}
\date{}
\usepackage{amsmath,amsfonts,array,amscd,amsthm,amssymb,stackrel,verbatim,wrapfig,subfig,float}
\usepackage{enumerate}
\usepackage{url}
\usepackage{tikz}
\usepackage[all]{xy}
\usetikzlibrary{arrows,matrix}
\usetikzlibrary{shapes.geometric}
\usetikzlibrary{decorations.markings} 

\tikzset{
  vertice/.style={circle,draw=black},
  decoration={markings,mark=at position 0.5 with {\arrow{>}}}
}

\newcommand{\wis}[1]{{\text{\em \usefont{OT1}{cmtt}{m}{n} #1}}}
\newcommand{\C}{\mathbb{C}}
\newcommand{\N}{\mathbb{N}}
\newcommand{\Z}{\mathbb{Z}}

\newcommand{\PP}{\mathbb{P}}

\theoremstyle{plain}
\newtheorem{theorem}{Theorem}[section]

\newtheorem{proposition}[theorem]{Proposition}
\newtheorem{corollary}[theorem]{Corollary}
\newtheorem{remark}[theorem]{Remark}

\DeclareMathOperator{\Ext}{Ext}

\numberwithin{equation}{section}
\begin{document}
\sloppy

\begin{abstract}
The quantum plane $A=\mathbb{C}_\rho[x,y,z]$ with $\rho$ a root of unity has singularities in its representation variety $\wis{trep}_n~A$ and its center $\frac{\C[u,v,w,g]}{uvw-g^n}$. Using the technique of a noncommutative blow-up, we prove that this technique fails in contrast to the 3-dimensional Sklyanin algebras if we want to resolve the singularities in $\wis{trep}_n~A$. However, we will see that the singularity of the center in the origin can be made better using this technique.
\end{abstract}
\maketitle
\section{Introduction}
Some of the easiest algebras to work with in noncommutative algebraic geometry are the quantum spaces (or planes if you want to work projectively) $A=\mathbb{C}_\rho[x,y,z]$ with defining equations 
\[\begin{cases}
xy=\rho yx\\
yz=\rho zy\\
zx=\rho xz
\end{cases}
\]
These algebras are finite modules over their center if and only if $\rho$ is a root of unity, with the center in these cases generated by $x^n,y^n,z^n,xyz$. While these algebras are relatively easy to work with, they have a disadvantage concerning their (trace preserving) representation variety $\wis{trep}_n~A$: it is not smooth. The singular locus is given by all $\pi^{-1}(\mathfrak{m})$, where $\mathfrak{m} \in \wis{max}~Z(A)$ is a singularity of the center and $\xymatrix{\wis{trep}_n~A \ar@{->>}[r]^\pi& \wis{max}~Z(A)}$ is the \textit{GIT}-quotient map.

In this paper there will be given a review concerning the representations, local quivers and singularities of the center of these algebras. We will always assume that $(n,3)=1$ as this will ensure that there will be fat point modules in addition to point modules.

One of the tools available in commutative algebraic geometry to resolve singularities is the use of blow-ups. Similar to \cite{BocklandtSymens}, where the quantum plane $\mathbb{C}_{-1}[u,v]$ was blown-up in the unique singularity, and \cite{DeLaetLeBruyn}, where the 3-dimensional Sklyanin algebras where blown-up in their unique singularity, we will use the construction of a noncommutative blow-up in section 6 to define an algebra $B = A \oplus I t \oplus I^2 t^2\oplus \ldots$ with $I = (x,z)$ and $I=(x,y,z)$. The first main theorem will be

\begin{theorem}
For the blow-up algebra $B = A \oplus I t \oplus I^2 t^2\oplus \ldots$, with $I = (x,z)$, the exceptional locus above a point $(0,b,0,0) \in \wis{max}~Z(A),b \neq 0$ in the partial resolution
\[
\xymatrix{\wis{trep}^{ss}_n~B//\wis{PGL}_n \times \C^* \ar@{->>}[r]
 & \wis{max}~Z(A)} \]
 is smooth, except for 2 points, where the singularity type is given by $\C \times \C^2/\Z_n$.
\end{theorem}

However, while the use of a noncommutative blow-up works in the case of the 3-dimensional Sklyanin algebras that are finite over their center (see \cite{DeLaetLeBruyn}), this does not do anything considering the representation variety. As $B$ is a finite module over its center, it will define a coherent sheaf $\mathcal{B}$ of Cayley-Hamilton algebras over $\wis{proj}~Z(B)$, but the sections over affine open subsets will still be graded noncommutative algebras. Therefore, we can never hope to find resolutions of $\wis{trep}_n~A$, as we will always find sections of $\mathcal{B}$ that have a bad singularity at the origin.
The second main result will be 

\begin{theorem}
For the blow-up algebra $B = A \oplus I t \oplus I^2 t^2\oplus \ldots$, with $I = (x,y,z)$, the exceptional locus above the point $(0,0,0,0) \in \wis{max}~Z(A)$ in the partial resolution
\[
\xymatrix{\wis{trep}^{ss}_n~B//\wis{PGL}_n \times \C^* \ar@{->>}[r]
 & \wis{max}~Z(A)} \]
 is smooth, except for the union of 3 lines intersecting 2 by 2. On regular points of these lines, the singularity type is $\C \times \C^2/\Z_n$.
\end{theorem}

\section{Review of known results}
The results in this section are well known in any literature regarding the quantum planes, see for example \cite{ArtinQuantumPlanes}.
\begin{theorem}
The algebra $A=\mathbb{C}_\rho[x,y,z]$ is an Artin-Schelter regular algebra with Hilbert series $\frac{1}{(1-t)^3}$ and is a finite module over its center if and only if $\rho$ is a root of unity. In this case and if $\rho \neq 1$, the center of $A$ is generated by 4 homogeneous elements $x^n,y^n,z^n,xyz$ with $n$ the order of $\rho$ in $\mu_\infty$.
\end{theorem}
It immediately follows that $Z(A)$ is isomorphic to $\mathbb{C}[u,v,w,g]/(uvw-g^n)$, which is an integral domain. This implies that
\begin{theorem}
The center of $A$ is the coordinate ring of $\C^3/G$, where $G$ is the subgroup of $\Z_n^3$ defined as the kernel of the map
$$
\xymatrix{\Z_n^3 \ar@{->>}[r]& \Z_n}, (a,b,c) \mapsto a+b+c
$$
\end{theorem}
\begin{proof}
Define the action of $G$ on $\C^3 = \C x + \C y+ \C z$ as 
$$(a,b,c) \cdot x = \rho^a x, (a,b,c) \cdot y = \rho^b y, (a,b,c) \cdot z = \rho^c z
$$
We will prove that $\C[\C^3/G]= \C[x,y,z]^G$ is generated by $x^n,y^n,z^n,xyz$, from which the claim will follow. It is clear that these 4 elements are $G$-invariant. Every monomial $x^ky^lz^m$ is a stable vector space for the action of $G$ and we may therefore assume that the monomial $x^ky^lz^m$ is fixed. Dividing by $xyz$, we may assume that one of the $k,l,m$ is 0, for example $m=0$. If $k$ is not divisible by $n$, the element $(1,0,-1) \in G$ will not fix $x^k y^l$, so $k$ is divisible by $n$. The same reasoning works for $l$ by using $(0,1,-1)$ and therefore $x^k y^l \in \C[x^n,y^n,z^n,xyz]$.
\end{proof}
Due to Galois descent and the fact that $Z(A)$ is integrally closed, we have a surjective trace map $\xymatrix{A \ar@{->>}[r]^{tr} & Z(A)}$ of degree 0 which turns $A$ into a graded Cayley-Hamilton algebra, that is, the couple $(A,tr)$ satisfies the following properties for all $a,b \in A$
\begin{itemize}
\item{$tr(ab) = tr(ba)$}
\item{$tr(1) = n$}
\item{$\chi_{n,a}(a) = 0$ with $\chi_{n,a}(X)$ the formal degree $n$ Cayley-Hamilton polynomial expressed in the traces of powers of $a$}
\end{itemize}
The object we want to study is the trace preserving representation variety $\wis{trep}_n~A$, which parametrizes all representations $A \stackrel{\phi}{\rightarrow} \wis{M}_n(\C)$ such that $\phi(tr(a))  = Tr(\phi(a))$, where $ \wis{M}_n(\C) \stackrel{Tr}{\longrightarrow} \mathbb{C}$ is the usual trace map. We have a natural action of $\wis{PGL}_n$ on this variety, given by conjugation. It follows from \cite{LBBook} that the \textit{GIT}-quotient $\wis{trep}_n~A/\wis{PGL}_n$ is isomorphic to $\wis{max}~Z(A)$, we will denote the corresponding quotient map by $\pi$. In order to describe $\wis{trep}_n~A$, it is useful to work with the representations of the Heisenberg group of order $n^3$. The description of this group by generators and relations is given by
\[
\langle e_1, e_2 | [e_1,e_2]e_1=e_1 [e_1,e_2], [e_1,e_2]e_2 = [e_1,e_2]e_2, e_1^n = e_2^n = 1 \rangle
\]
and this group has $\varphi(n)$ simple representations of dimension $n$, with $\phi$ the Euler totient function. These representations are determined by a primitive $n$th root of unity $\rho^k$ with $(k,n)=1$ and are given in the following way: let $V_{\rho^k} = \oplus_{i=0}^{n-1}\mathbb{C}x_i$, then the action of $H_n$ is defined by
\begin{align*}
e_1 \cdot x_i = x_{i-1} , e_2 \cdot x_i = \rho^{ki} x_i, i=0\ldots n-1
\end{align*}
indices taken $\bmod n$. For the corresponding group morphism $H_n \stackrel{\psi_k}{\rightarrow} \wis{GL}_n$, it is easily checked that $[e_1,e_2]$ is send to $\rho^k I_n$. Using these representations, one finds that
\begin{theorem}
Let $(a,b,c)\in \wis{max}~\mathbb{C}[x^n,y^n,z^n]$, $(n,2)=1$ such that $abc \neq 0$ and choose $n$th roots of $a,b,c$, say $\alpha,\beta,\gamma$, then the $n$ corresponding orbits of semi-simple representations lying above this point have representatives of the form \\ $A \xrightarrow{\phi_{(\alpha,\beta,\gamma)}^j} \wis{M}_n(\C), j = 0,\ldots,n-1$ determined by
\[
\phi_{(\alpha,\beta,\gamma)}^j(x) = \alpha \psi_1(e_1),\phi_{(\alpha,\beta,\gamma)}^j(y) = \beta \psi_1(e_2),\phi_{(\alpha,\beta,\gamma)}^j(z) = \gamma \rho^j \psi_1(e_2^{-1}e_1^{-1})
\]
If one of them is 0, for example $c=0$, then there is a unique orbit lying above the corresponding point with representative $A \xrightarrow{\phi} \wis{M}_n(\C), j = 0,\ldots,n-1$ determined by
\[
\phi(x) = \alpha \psi_1(e_1),\phi(y) = \beta \psi_1(e_2),\phi(z) = 0
\]
If 2 of them are 0, for example $a=c=0$, then there is a unique orbit of semi-simple representations lying above the corresponding point determined by 
\[
\phi(x) = 0,\phi(y) = \beta \psi_1(e_2),\phi(z) = 0
\]
When $(a,b,c) = (0,0,0)$, then the corresponding semi-simple representation is the trivial representation.
\label{th:normalform}
\end{theorem}
We will call a semi-simple representation of $A$ given in standard form if it is determined by a triple $(\alpha e_1,\beta e_2,\gamma e_2^{-1}e_1^{-1})$. This standard form is not necessarily unique.

The fact that for every point of the open set $abc \neq 0$ there are $n$ different simple representations lying above the corresponding point of $\wis{max}~\mathbb{C}[x^n,y^n,z^n]$ is a consequence of $Z(A)$ being an extension of degree $n$ over $\mathbb{C}[x^n,y^n,z^n]$. In order to find the representations in the case that $(2,n)=2$, one has to take $\gamma=\sqrt[n]{-c}$ instead of $\gamma=\sqrt[n]{c}$ as we have that
$$
(e_2^{-1} e_1^{-1})^n = [e_1,e_2]^{\frac{n(n-1)}{2}}
$$
and $\frac{n(n-1)}{2}$ is only divisible by $n$ if $2$ does not divide $n$.

Using these representations of $A$, it is easy to see that the non-Azumaya locus of $Z(A)$ is determined by the 3 lines 
\begin{itemize}
\item $x^n=0, y^n = 0$,
\item $x^n=0,z^n=0$ and
\item $y^n=0,z^n=0$.
\end{itemize}  
Another way to describe the Azumaya locus is
\begin{corollary}
A point $p \in \wis{max}~Z(A)$ belongs to $\wis{Azu}_n~A$ if and only if $\pi^{-1}(p)$ is isomorphic to $\mathcal{O}(V_\rho) \subset \wis{rep}_n~H_n$.
\end{corollary}
As the representation $V_\rho$ of $H_n$ is faithful, we will identify $e_i$ and the matrix $\psi(e_i)$ from now on.
\section{$\wis{trep}_n~A$}
In this section, we look at the singularities of $\wis{trep}_n~A$. In contrast to the Sklyanin algebras, where there is only a single singularity corresponding to the trivial representation, there are more singularities to consider.

The dimension of $\wis{trep}_n~A$ is $n^2+2$, which follows from the fact that the quotient map $\xymatrix{\wis{trep}_n(A)\ar@{->>}[r]^{\pi}&\wis{max}~Z(A)}$ is a principal $\wis{PGL}_n$-fibration over $\wis{Azu}_n~A$, which is an open subset of $\wis{max}~Z(A)$.
\begin{theorem}
$\wis{trep}_n~A$ is singular in $p$ if and only if $\wis{max}~Z(A)$ is singular in $\pi(p)$.
\end{theorem}
\begin{proof}
We already know that the singularities of the center are given by the 3 lines $x^n=0, y^n = 0$, $x^n=0,z^n=0$ and $y^n=0,z^n=0$. From this it follows that for every point $q$ on these lines, $\wis{trep}_n~A$ has to be singular in $\pi^{-1}(q)$, since the only central singularities possible are isolated (for a 3-dimensional center, only the conifold singularity is possible). For any regular point $q$ of $\wis{max}~Z(A)$ we have that $q \in \wis{Azu}_n~A$ and so for every point in $\pi^{-1}(q)$ the dimension is $n^2+2$. The claim follows.
\end{proof}
We immediately find that the local quiver (see \cite{LBBook}) in a point of the Azumaya locus looks like
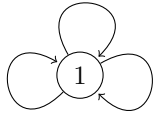
\begin{center}
\begin{tikzpicture}
   \node[vertice] (a) at (0,0) {1};
   
   \path[->] (a) edge[out=30,in=-45,loop] (a);
   \path[->] (a) edge[out=120,in=45,loop] (a);
   \path[->] (a) edge[out=225,in=150,loop] (a);   
\end{tikzpicture}
\end{center}
For a point belonging to the Azumaya locus we can even find a good description of $N=T_p \wis{trep}_n~A/T_p \mathcal{O}(p)$ as a subspace of $T_p \wis{trep}_n~A \subset \wis{M}_n(\C)^{\oplus 3}$.
\begin{proposition}
The normal $N = \Ext_A^{tr}(M,M)$ for a simple module $M$ given by a triple $a e_1, b e_2, c e_2^{-1}e_1^{-1}$  is isomorphic to the subspace of $T_M \wis{trep}_n~A \subset \wis{M}_n(\C)^{\oplus 3}$ generated by the triples $(e_1,0,0),(0,e_2,0),(0,0,e_2^{-1}e_1^{-1})$.
\label{prop:normaldecomposition}
\end{proposition}
\begin{proof}
For any $(a,b,c) \in \C^3$ such that the matrices $a e_1, b e_2, c e_2^{-1}e_1^{-1}$ form a simple representation, we have that for any choice of $a_1,b_1,c_1 \in \C$ the matrices
$$
a e_1+ \varepsilon a_1 e_1, b e_2+ \varepsilon b_1 e_2, c e_2^{-1}e_1^{-1}+ \varepsilon c_1 e_2^{-1}e_1^{-1}
$$
form a representation of $A$ over the dual numbers $\C[\varepsilon]/(\varepsilon^2)$. This means that  $(e_1,0,0),(0,e_2,0),(0,0,e_2^{-1}e_1^{-1}) \in T_M \wis{trep}_n~A$. The fact that these 3 matrix triples are linearly independent if we divide out $T_M \mathcal{O}(M)$ follows from the fact that, if we take $a_1,b_1,c_1$ small enough, the matrices 
$$
a e_1+ a_1 e_1, b e_2+ b_1 e_2, c e_2^{-1}e_1^{-1}+ c_1 e_2^{-1}e_1^{-1}
$$
are again a simple representation of $A$ non-isomorphic to $M$.
\end{proof}
In the point corresponding to the trivial representation of $A$, we have
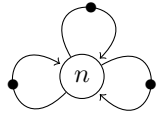
\begin{center}
\begin{tikzpicture}
   \node[vertice] (a) at (0,0) {$n$};
   
   \path[->] (a) edge[out=30,in=-45,loop] node {$\bullet$} (a);
   \path[->] (a) edge[out=120,in=45,loop] node {$\bullet$} (a);
   \path[->] (a) edge[out=225,in=150,loop] node {$\bullet$} (a);   
   
\end{tikzpicture}
\end{center}
where the marked arrows correspond to matrices with trace 0. This is exactly the same as in the case of the Sklyanin algebras. The main difference with the Sklyanin algebras lies in the other singularities: consider for example the semi-simple representation $\psi$ given by sending $x$ and $z$ to $0$ and sending $y$ to $b e_2$. In order to find the tangent space to the corresponding point of $\wis{trep}_n~A$, we need to take traceless matrices $A,B,C$ and look at representations $f_\psi$ over $\C[\varepsilon]/(\varepsilon^2)$ of the form
\[
f_\psi(x)= \varepsilon A, f_\psi(y)=b e_2 + \varepsilon B, f_\psi(y)= \varepsilon C
\]
from which it follows that
\[
A=\begin{bmatrix}
0      & a_1       & 0        & \ldots & 0\\
0      & 0         & a_2      & \ddots & 0\\
\vdots & \vdots    & \ddots   & \ddots & \vdots \\
0      & 0         & \ldots   & 0      & a_{n-1}\\
a_n    & 0         & \ldots   & 0      & 0
\end{bmatrix},
C = 
\begin{bmatrix}
0      & 0        & \ldots   & 0      & c_n\\
c_1    & 0        & \ldots   & 0 & 0\\
\vdots & c_2      & \ddots   & \ddots & \vdots \\
0      & \vdots   & \ddots   & 0      & 0\\
0      & 0        & \ldots   & c_{n-1}& 0
\end{bmatrix}
\]
and $B$ a matrix for which $tr(B^k)=0$ for $1 \leq k \leq n-1$. This implies that the dimension of the tangent space in this point is $n^2-n+1$. In order to calculate the tangent space to the orbit of this point, we need to see what the action of $I_n+ \varepsilon T$ on the triple $(0,b e_2,0)$ is by conjugation with $T \in \wis{M}_n(\C)$. It follows that the tangent space to the orbit is given by triples $(0,D,0)$ with $D$ a $n\times n$-matrix with zeros on the diagonal. Consequently, we can take representatives for elements of $N$ to be given by triples of matrices $(A,B,C)$ such that $A$ and $C$ are as above and $B = t e_2$. From this we deduce that the local quiver in this point is given by
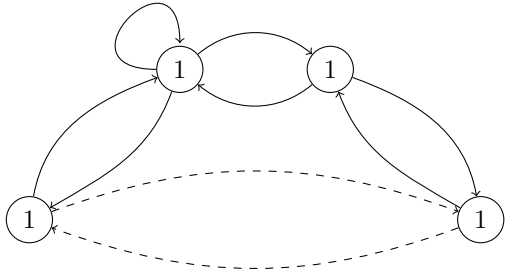
\begin{center}
\begin{tikzpicture}
   \node[vertice] (a) at (0,0) {$1$};
   \node[vertice] (b) at (2,0) {$1$};
   \node[vertice] (c) at (4,-2) {$1$};
   \node[vertice] (d) at (-2,-2) {$1$};
   
   \path[->] (a) edge[out=40,in=140] (b);
   \path[->] (b) edge[out=-140,in=-40] (a);
   \path[->] (b) edge[out=-20,in=100] (c);
   \path[->] (c) edge[out=150,in=-70] (b);
     
   \path[->] (a) edge[out=-110,in=30] (d);
   \path[->] (d) edge[out=80,in=-160] (a);   
   
   \path[->,dashed] (d) edge[out=20,in=160] (c);
   \path[->,dashed] (c) edge[out=-160,in=-20] (d);
   
   \path[->] (a) edge[out=180,in=90,loop] (a);
   
\end{tikzpicture}
\end{center}
In short, we have 2 classes of singularities in $\wis{trep}_n~A$ to study.

For a semi-simple $n$-dimensional representation $M= \oplus_{i=1}^k S_i^{e_i}$ with corresponding maximal ideal $\mathfrak{m}$ of the center  of $A$ one can associate a \textit{defect} against $A$ being Cayley-smooth in $\mathfrak{m}$.
\[
\wis{defect}_\mathfrak{m}(A)=\dim_\C \Ext_A^{tr}(M,M)+(n^2-\sum_{i=1}^k e_i^2)- \dim \wis{trep}_n~A
\]
From the local quivers, we obtain
\begin{theorem}
The defects for any semi-simple $n$-dimensional trace preserving representation are given by
\begin{itemize}
\item when $\mathfrak{m}$ does not belong to the singular locus of the center, $\wis{defect}_\mathfrak{m}(A)=0$
\item if $\mathfrak{m}$ belongs to the singular locus but is not equal to $(x^n,y^n,z^n,xyz)$, we have $\wis{defect}_\mathfrak{m}(A)=n-1$
\item if $\mathfrak{m}=(x^n,y^n,z^n,xyz)$, we have $\wis{defect}_\mathfrak{m}(A)=2n^2-5$
\end{itemize}
\end{theorem}

\section{$\wis{Proj}~A$}
We have found that there are 2 types of points in $\wis{Azu}_n~A$, determined by $\phi(xyz)$ equal to 0 or not. In order to better describe the difference between these 2 kind of points in $\wis{Azu}_n~A$, one has to use the fact that $A$ is graded. We will use \cite{BocklandtSymens} to explain the difference.

We will no longer work with $\wis{trep}_n~A$, but with the semi-stable representations $\wis{trep}^{ss}_n~A$, where for every positively graded algebra $A'$ $\wis{trep}^{ss}_n~A'$ is defined as
$$
\wis{trep}^{ss}_n~A'=\left\{ \phi \in  \wis{trep}_n~A': \exists k \geq 1, f \in Z(A')_k: \phi(f) \neq 0 \right\}
$$

There is a natural action of $\wis{PGL}_n\times \C^*$ on $\wis{trep}^{ss}_n~A'$, with the $\wis{PGL}_n$-action given by conjugation and the $\C^*$-action coming from the $\C^*$-action on $A'$, which follows from the natural gradation. It is clear that these 2 actions commute.

Let us return to the case $A=\mathbb{C}_\rho[x,y,z]$. Suppose that $P \in \wis{Azu}_n~A$, let $\xymatrix{A\ar@{->>}[r]^\phi & \wis{M}_n(\C)}$ be the corresponding algebra epimorphism. The $\C^*$-orbit of $M$ defines a graded module $F$ with corresponding algebra morphism
$$
\xymatrix{
A \ar[r] & \wis{M}_n(\C[t,t^{-1}])(a_1,\ldots,a_{n})
}
$$
where for a graded ring $R$ the gradation of $\wis{M}_n(R)(a_1,\ldots,a_{n})$ is defined by
\[
\wis{M}_n(R)(a_1,\ldots,a_{n})_i=
\begin{bmatrix} 
R_i & R_{i-a_1+a_2} & \hdots & R_{i-a_1+a_n} \\
R_{i-a_2+a_1} & R_i & \hdots & R_{i-a_2+a_n} \\
\vdots & \vdots & \ddots & \vdots \\
R_{i-a_n+a_1} & R_{i-a_n+a_2} & \hdots & R_i
\end{bmatrix} \]
If $\phi(xyz)\neq 0$, then $\phi$ induces an algebra morphism from $A_{(xyz)}$ to $\wis{M}_n(\C)$. It then follows that $t$ is the image of a degree 1 central element of $A_{(xyz)}$ and from \cite{LeBruyn1996d} it follows that all $a_i$ are equal to $0$. In the terminology of \cite{ArtinQuantumPlanes}, this means that the $\wis{PGL}_n\times \C^*$-action on $M$ defines a fat-point module of degree $n$, that is, a 1-critical module with Hilbert series $\frac{n}{1-t}$.

However, if $\phi(xyz)= 0$, then $t$ is the image of a degree $n$ central element of $A/(xyz)$. It follows from \cite{LeBruyn1996d} that $a_i = i-1, 1 \leq i \leq n$. In the terminology of \cite{ArtinQuantumPlanes}, the graded module $F$ is not a fat-point module, but instead a direct sum of $n$ point modules, which are 1-critical graded modules with Hilbert series $\frac{1}{1-t}$. 

\cite{BocklandtSymens} sets this in a \textit{GIT}-setting: this difference is given by the fact that for a degree $n$ fat-point module $F$ and a chosen representative $M$ of the $\C^*$-orbit determined by $F$, the stabilizer in $\wis{PGL}_n\times \C^*$ of $M$ is trivial, while for the simple modules for which $xyz$ is in the kernel of the algebra map, the stabilizer is not trivial. According to \cite{BocklandtSymens}, this stabilizer for a simple module is always a finite cyclic subgroup of $\wis{PGL}_n\times \C^*$.

In our case, the stabilizer for the simple representation lying above $P = (a,b,0,0)$ (with $a \neq 0 \neq b$) written like in theorem \ref{th:normalform} can be calculated to be the group $\langle (e_1 e_2^{-1}, \rho) \rangle \subset \wis{PGL}_n\times \C^*$. The information of this stabilizer can be decoded using \textit{weighted} quiver settings, that is, associating to each arrow a weight which decodes the decomposition of $N$ in simple representations of $\Z_n$ when $Stab(M) \cong \Z_n$.

\begin{proposition}
We have for $P \in \wis{Azu}_n~A$ and corresponding irreducible representation $M$
\begin{itemize}
\item 
If $P \notin \mathbf{V}(xyz)$, then the normal space $N$ is as $Stab(M)=\C^* I_n$ representation given by
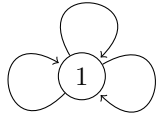
\begin{center}
\begin{tikzpicture}
   \node[vertice] (a) at (0,0) {$1$};
   
   \path[->] (a) edge[out=30,in=-45,loop] (a);
   \path[->] (a) edge[out=120,in=45,loop] (a);
   \path[->] (a) edge[out=225,in=150,loop] (a);   
\end{tikzpicture}
\end{center}
\item If $P \in \mathbf{V}(xyz)$, then the normal space $N$ is as $Stab(M)=\C^* I_n \times \Z_n$-representation given by
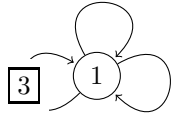
\begin{center}
\begin{tikzpicture}
   \node[vertice] (a) at (0,0) {$1$};
   \tikzset{every node/.style={fill=white}} 
   \path[->] (a) edge[out=30,in=-45,loop]  (a);
   \path[->] (a) edge[out=120,in=45,loop] (a);
   \path[->] (a) edge[out=225,in=150,loop]node{$\boxed{3}$} (a);   
\end{tikzpicture}
\end{center}
\end{itemize}
\end{proposition} 
\begin{proof}
We have found that a basis for the normal space $N$ in an Azumaya point is determined by the 3-dimensional subspace of $\wis{M}_n(\C)^{\oplus 3}$ generated by $(e_1,0,0),(0,e_2,0),(0,0,e_2^{-1}e_1^{-1})$. Calculating the action of the stabilizer subgroup $\langle (e_1 e_2^{-1}, \rho) \rangle$, we find
\begin{gather*}
\rho e_2 e_1^{-1} e_1 e_1 e_2^{-1} = e_1 \\
\rho e_2 e_1^{-1} e_2 e_1 e_2^{-1} = e_2 \\
\rho e_2 e_1^{-1} e_2^{-1}e_1^{-1} e_1 e_2^{-1} = \rho^3 e_2^{-1}e_1^{-1}
\end{gather*}
from which the weighted quiver setting follows.
\end{proof}
One can do the same for the non-trivial singular points in the center, although in this case the $\wis{PGL}_n\times \C^*$-stabilizer will be infinite. Let $M$ be a singular point of $\wis{trep}_n^{ss}~A$ with stabilizer $Stab(M)$. From \cite{BocklandtSymens} it follows that $Stab(M)$ is given by $(\C^*)^n \rtimes_\psi \Z/n\Z$ with $\psi$ a finite order automorphism of the local quiver, which can be written as $\psi = w \phi$ with $w$ a weight and $\phi$ a twist that commute. The finite group $\langle (e_1, \rho) \rangle \subset \wis{PGL}_n\times \C^*$ again stabilizes $M$ as can be easily calculated. From this subgroup of $Stab(M)$, we deduce
\begin{theorem}
For $(x^n,y^n,z^n,xyz)\neq P \in \wis{sing}(Z(A))$, the twisted weighted quiver setting is determined by
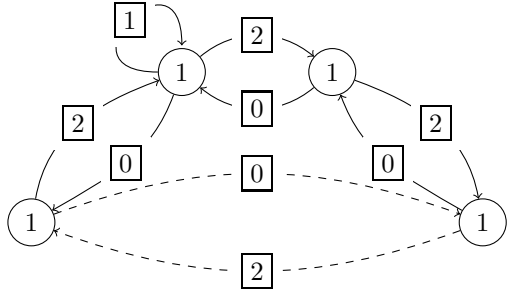
\begin{center}
\begin{tikzpicture}
   \node[vertice] (a) at (0,0) {$1$};
   \node[vertice] (b) at (2,0) {$1$};
   \node[vertice] (c) at (4,-2) {$1$};
   \node[vertice] (d) at (-2,-2) {$1$};
   
\tikzset{every node/.style={fill=white}}   
   
   \path[->] (a) edge[out=40,in=140]node{$\boxed{2}$} (b);
   \path[->] (b) edge[out=-140,in=-40]node{$\boxed{0}$} (a);
   \path[->] (b) edge[out=-20,in=100]node{$\boxed{2}$} (c);
   \path[->] (c) edge[out=150,in=-70]node{$\boxed{0}$} (b);
     
   \path[->] (a) edge[out=-110,in=30]node{$\boxed{0}$} (d);
   \path[->] (d) edge[out=80,in=-160]node{$\boxed{2}$} (a);   
   
   \path[->,dashed] (d) edge[out=20,in=160]node{$\boxed{0}$} (c);
   \path[->,dashed] (c) edge[out=-160,in=-20]node{$\boxed{2}$} (d);
   
   \path[->] (a) edge[out=180,in=90,loop]node{$\boxed{1}$} (a);
   
\end{tikzpicture}
\end{center}
\end{theorem}
\begin{proof}
$N$ decomposes as a triple of matrices $(A,B,C)$ such that (up to cyclic permutation) 
\[
A=\begin{bmatrix}
0      & a_1       & 0        & \ldots & 0\\
0      & 0         & a_2      & \ddots & 0\\
\vdots & \vdots    & \ddots   & \ddots & \vdots \\
0      & 0         & \ldots   & 0      & a_{n-1}\\
a_n    & 0         & \ldots   & 0      & 0
\end{bmatrix},
C = 
\begin{bmatrix}
0      & 0        & \ldots   & 0      & c_n\\
c_1    & 0        & \ldots   & 0 & 0\\
\vdots & c_2      & \ddots   & \ddots & \vdots \\
0      & \vdots   & \ddots   & 0      & 0\\
0      & 0        & \ldots   & c_{n-1}& 0
\end{bmatrix}
\]
and $B = t e_2$. It follows that
\begin{gather*}
(e_1 e_2^{-1}, \rho) \cdot A = \begin{bmatrix}
0      & a_n       & 0        & \ldots & 0\\
0      & 0         & a_1      & \ddots & 0\\
\vdots & \vdots    & \ddots   & \ddots & \vdots \\
0      & 0         & \ldots   & 0      & a_{n-2}\\
a_{n-1}& 0         & \ldots   & 0      & 0
\end{bmatrix} \\
(e_1 e_2^{-1}, \rho) \cdot B = \rho b \begin{bmatrix}
\rho^{-1}    & 0         & 0        & \ldots & 0\\
0      & 1       & 0        & \ddots & 0\\
\vdots & \vdots    & \ddots   & \ddots & \vdots \\
0      & 0         & \ldots   & \rho^{-3}& 0\\
0      & 0         & \ldots   & 0      &\rho^{-2}
\end{bmatrix}\\
(e_1 e_2^{-1}, \rho) \cdot C =
\rho^2 \begin{bmatrix}
0      & 0        & \ldots   & 0      & c_{n-1}\\
c_n    & 0        & \ldots   & 0 & 0\\
\vdots & c_1      & \ddots   & \ddots & \vdots \\
0      & \vdots   & \ddots   & 0      & 0\\
0      & 0        & \ldots   & c_{n-2}& 0
\end{bmatrix}
\end{gather*}
From this it follows that the corresponding automorphism on the quiver $Q$ is the composition of a twist given by cyclic permutation on the vertices and the arrows and a weight defined by the claimed weights.
\end{proof}
In order to find the twisted weighted quiver setting, we need to find $N_{\wis{PGL}_n \times \C^*}$. We know from \cite{BocklandtSymens} that $\dim N_{\wis{PGL}_n} = 1 +\dim N_{\wis{PGL}_n \times \C^*}$. As the action of $\wis{PGL}_n$ and $\C^*$ commute, it is enough to consider the action of $\C^*$. We have
$$
(1+\varepsilon t) (0,b e_2,0) = (0,b e_2,0) + \varepsilon(0,tb e_2,0)
$$
from which follows that the loop in the local quiver should be deleted to find $N_{\wis{PGL}_n \times \C^*}$. Summarizing, we have
\begin{theorem}
The \'etale local structure of the \textit{GIT}-quotient 
$$
\wis{trep}^{ss}_n~A//\wis{PGL}_n \times \C^*
$$
is given by the following twisted weighted quiver settings:
\begin{itemize}
\item If $P \notin \mathbf{V}(x^ny^nz^n)$, then we have
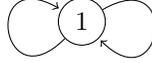
\begin{center}
\begin{tikzpicture}
   \node[vertice] (a) at (0,0) {$1$};
   
   \path[->] (a) edge[out=30,in=-45,loop] (a);
   \path[->] (a) edge[out=225,in=150,loop] (a);   
\end{tikzpicture}
\end{center}
\item If $P \in \mathbf{V}(x^ny^nz^n)$ but is not a singular point of this variety, we have 
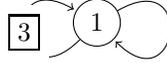
\begin{center}
\begin{tikzpicture}
   \node[vertice] (a) at (0,0) {$1$};
   \tikzset{every node/.style={fill=white}} 
   \path[->] (a) edge[out=30,in=-45,loop]  (a);
   \path[->] (a) edge[out=225,in=150,loop]node{$\boxed{3}$} (a);   
\end{tikzpicture}
\end{center}
\item If $P$ is one of the 3 singular points of $\mathbf{V}(x^ny^nz^n)$, then we have
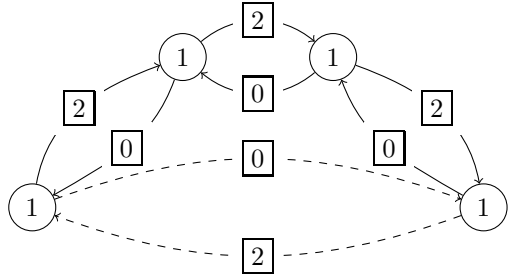
\begin{center}
\begin{tikzpicture}
   \node[vertice] (a) at (0,0) {$1$};
   \node[vertice] (b) at (2,0) {$1$};
   \node[vertice] (c) at (4,-2) {$1$};
   \node[vertice] (d) at (-2,-2) {$1$};
   
\tikzset{every node/.style={fill=white}}   
   
   \path[->] (a) edge[out=40,in=140]node{$\boxed{2}$} (b);
   \path[->] (b) edge[out=-140,in=-40]node{$\boxed{0}$} (a);
   \path[->] (b) edge[out=-20,in=100]node{$\boxed{2}$} (c);
   \path[->] (c) edge[out=150,in=-70]node{$\boxed{0}$} (b);
     
   \path[->] (a) edge[out=-110,in=30]node{$\boxed{0}$} (d);
   \path[->] (d) edge[out=80,in=-160]node{$\boxed{2}$} (a);   
   
   \path[->,dashed] (d) edge[out=20,in=160]node{$\boxed{0}$} (c);
   \path[->,dashed] (c) edge[out=-160,in=-20]node{$\boxed{2}$} (d);
   
   
\end{tikzpicture}
\end{center}
 
\end{itemize}
\end{theorem}
\section{Smoothness of $\mathcal{A}$}
This section uses the tools and definitions developed in \cite{LBBook}, chapter 5, section 5.4.

As $A$ is a finite module over its center, it defines a coherent sheaf $\mathcal{A}$ of algebras over $\wis{proj}~Z(A)=\PP^2$. On the affine open subset $\mathbb{X}(x^n)$, $\Gamma(\mathbb{X}(x^n),\mathcal{A})$ is defined as the following ring: as $(n,3)=1$, there exists a degree 1 central element in the graded localisation ring $Q^g_{x^n}(A)$. Therefore, we have
$$
Q^g_{x^n}(A) = (Q^g_{x^n}(A))_0[t,t^{-1}]
$$
and by definition $\Gamma(\mathbb{X}(x^n),\mathcal{A})=(Q^g_{x^n}(A))_0$. It is easy to see that $(Q^g_{x^n}(A))_0\cong \C_{\rho^3} [u,v]$ and therefore $\Gamma(\mathbb{X}(x^n),\mathcal{A})$ is an Auslander regular algebra of dimension 2 and consequently a maximal order. Just as in the Sklyanin case, we obtain
\begin{proposition}
$\mathcal{A}$ defines a coherent sheaf of Cayley-Hamilton maximal orders over $\PP^2$, which are Auslander regular domains of dimension 2.
\end{proposition}
\begin{remark}
This is one of the reasons why we need $(n,3)=1$: if $n$ were divisible by $3$, $\mathcal{A}$ wouldn't define a sheaf of Cayley-Hamilton algebras of degree $n$, but of $\tfrac{n}{3}$.
\end{remark}
As $\mathcal{A}$ defines a sheaf of maximal orders in a central simple algebra $\Sigma$ of degree $n$ over $\C(x,y) = \C(\PP^2)$, the Artin-Mumford exact sequence (see \cite{LBBook}) states that $\Sigma$ is defined by its ramification locus and a $\Z_n$-cover of this variety.

Unfortunately, where in the Sklyanin case the ramification locus is given by an elliptic curve $E'=E/\langle \tau \rangle$ which is smooth and from which follows that the corresponding sheaf is a sheaf of Cayley-smooth algebras, the same is not true for the quantum algebras. As we have seen before, the ramification locus is given by 3 lines in $\PP^2$, which intersect 2 by 2. Therefore, we will have singularities to consider. Let us look at $\mathbb{X}(x^n)$ and work out the ramification in this case.

We know that $\Gamma(\mathbb{X}(x^n),\mathcal{A})\cong \C_{\rho^3} [u,v]$, which is again a Cayley-Hamilton algebra of degree $n$ as we assumed that $(n,3)=1$. For this algebra it is known that the center is generated by $u^n,v^n$ and $\wis{max}~\C[u^n,v^n] = \C^2$. Considering trace preserving representations, we have
\begin{theorem}
\cite{LBBook} For the quantum plane $\C_{\rho^3} [u,v]$, we have the following kinds of semi-simple trace preserving representations for a point $\mathfrak{m}\in \wis{max}~\C[u^n,v^n]$
\begin{itemize}
\item If $\mathfrak{m}=(u^n-a,v^n-b)$ with $a \neq 0 \neq b$, then $\mathfrak{m} \in \wis{Azu}_n$.
\item If $\mathfrak{m}=(u^n,v^n-b)$ or $\mathfrak{m}=(u^n-a,v^n)$ with $a \neq 0 \neq b$, then the corresponding semi-simple representation is a direct sum of $n$ distinct 1-dimensional simple representations.
\item If $\mathfrak{m}=(u^n,v^n)$, then the corresponding semi-simple representation is the trivial representation with multiplicity $n$.
\end{itemize}
\end{theorem}
This of course also holds for the other affine opens $\mathbb{X}(y^n)$ and $\mathbb{X}(z^n)$. Analogous as for $A$, one finds
\begin{theorem}
\cite{LBBook} The local quiver settings for a point $\mathfrak{m}\in \wis{max}~\C[u^n,v^n]$ are given by
\begin{itemize}
\item If $\mathfrak{m}=(u^n-a,v^n-b)$ with $a \neq 0 \neq b$, then we have
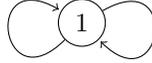
\begin{center}
\begin{tikzpicture}
   \node[vertice] (a) at (0,0) {$1$};
   \tikzset{every node/.style={fill=white}} 
   \path[->] (a) edge[out=30,in=-45,loop]  (a);
   \path[->] (a) edge[out=225,in=150,loop] (a);   
\end{tikzpicture}
\end{center}
\item If $\mathfrak{m}=(u^n,v^n-b)$ or $\mathfrak{m}=(u^n-a,v^n)$ with $a \neq 0 \neq b$, then we have
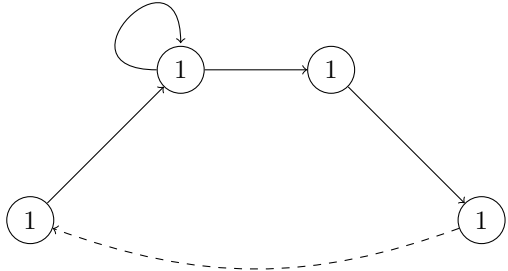
\begin{center}
\begin{tikzpicture}
   \node[vertice] (a) at (0,0) {$1$};
   \node[vertice] (b) at (2,0) {$1$};
   \node[vertice] (c) at (4,-2) {$1$};
   \node[vertice] (d) at (-2,-2) {$1$};
   
\tikzset{every node/.style={fill=white}}   
   
   \path[->] (a) edge[out=0,in=180] (b);
   \path[->] (b) edge[out=-45,in=135](c);
     
   \path[->] (d) edge[out=45,in=-135] (a);
   
   \path[->,dashed] (c) edge[out=-160,in=-20] (d);
   
   \path[->] (a) edge[out=180,in=90,loop] (a);
   
\end{tikzpicture}
\end{center}
\item If If $\mathfrak{m}=(u^n,v^n)$, then we have 
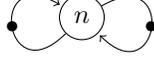
\begin{center}
\begin{tikzpicture}
   \node[vertice] (a) at (0,0) {$n$};
   
   \path[->] (a) edge[out=30,in=-45,loop] node {$\bullet$} (a);
   \path[->] (a) edge[out=225,in=150,loop] node {$\bullet$} (a);   
   
\end{tikzpicture}
\end{center}
\end{itemize}
\label{th:sheaflocal}
\end{theorem}
The same is true for $\mathbb{X}(y^n)$ and $\mathbb{X}(z^n)$, so it follows that
\begin{theorem}
$\mathcal{A}$ is not a sheaf of Cayley-Smooth algebras over $\PP^2$. We have for the marked local quiver setting
\begin{itemize}
\item If $P \notin \mathbf{V}(x^ny^nz^n)$, then $\mathcal{A}$ is Azumaya in $P$ and the local quiver setting is of the first kind of theorem \ref{th:sheaflocal}.
\item If $P \in \mathbf{V}(x^ny^nz^n)$, then $P$ belongs to the ramification locus. If $P$ is not one of the singular points, then there are $n$ 1-dimensional representations of $\mathcal{A}$ lying above $P$ and the quiver setting is of the second kind of theorem \ref{th:sheaflocal}.
\item If $P$ is one of the singular points of $\mathbf{V}(x^ny^nz^n)$, then we have a unique 1-dimensional representation of $\mathcal{A}$ lying above $P$ and the local quiver is of the third kind.
\end{itemize}
\end{theorem}

\section{The noncommutative blow-up}
\subsection{Blow-up of a line} We will first describe the blow-up along the singular part of the center defined by the line $x^n=0,z^n=0$. In this case, the blow-up algebra $B$ is defined by the subalgebra of $A[t]$ with $t$ central generated by $A$ and $X=xt,Z=zt$. Let $I=(x,z)\triangleleft A$, then the gradation on $B$ is defined by
$$
B = A \oplus It \oplus I^2t^2 \oplus \ldots
$$
or in other words, $A$ is given degree 0 and $t$ degree 1. The relations of $B$ are given by the quantum relations (the relations of $A$), $ZX=\rho XZ$ and commutation relations such as $xX=Xx,xZ=Xz$, etc. It is also clear that $B$ is a Cayley-Hamilton algebra of degree $n$. The inclusion of $A$ in $B$ defines an epimorphism $\xymatrix{\wis{trep}^{ss}_n~B \ar@{->>}[r]&\wis{trep}^{ss}_n~A}$ which by composition gives an epimorphism $\xymatrix{\wis{trep}^{ss}_n~B \ar@{->>}[r]^\Pi& \wis{max}~Z(A)}$.

We have
\begin{proposition}
The dimension of $\wis{trep}_n~B$ is $n^2+3$.
\end{proposition}
\begin{proof}
For every maximal ideal $\mathfrak{m} \in \wis{Azu}_n~A$, the localisation of $B$ at $\mathfrak{m}$ is equal to
$$
B_\mathfrak{m}=A_\mathfrak{m}[t,t^{-1}]
$$
This means that $B_\mathfrak{m}$ is an Azumaya algebra over $Z(A)_\mathfrak{m}[t,t^{-1}]$. We know that the smooth locus of $\wis{max}~Z(A)$ equals $\wis{Azu}_n~A$, therefore the dimension of the tangent space in each point lying above $\mathfrak{m}$  in $\wis{trep}^{ss}_n~B$ is
$$
\dim T_\phi \wis{trep}^{ss}_n~B = n^2-1+ \dim T_\mathfrak{m} \wis{max}~Z(A) +1 
$$
with the $+1$ coming from the fact that $Z(A)[t,t^{-1}]=Z(A) \otimes \C[t,t^{-1}]$. This subset of $ \wis{trep}^{ss}_n~B$ is open and so the claim follows.
\end{proof}

\begin{proposition}
$\wis{trep}^{ss}_n~B$ is smooth in the inverse image of $\Pi$ for every point $P$ on the line $\mathbf{V}(x^n,z^n)$ except the point $P = (0,0,0,0)$.
\end{proposition}
\begin{proof}
Let $\phi \in \wis{trep}^{ss}_n~B$ and assume $\phi(X)$ is invertible. By assumption $\phi(y)$ is also invertible and we may assume that the triple $(\phi(X),\phi(y),\phi(Z))$ is in standard form. We need to find the dimension of the solution set of 5-tuples of matrices $(A,B,C,D,E)$ such that 
$$
\phi(X) + \varepsilon A, \phi(Z) + \varepsilon B, \varepsilon C, \phi(y)+\varepsilon D, \varepsilon E
$$
is a representation of the blow-up algebra $B$ over the dual numbers. The subalgebra of $B$ generated by $X,Z$ and $y$ is isomorphic to $A$ itself and $\phi$ induces a simple representation of this subalgebra, so we know that $A,B$ and $D$ depend on $n^2+2$ parameters. From the relation $Xz = xZ$ it follows that $\phi(X) E = C \phi(Z)$. Now, two things can happen:
\begin{itemize}
\item $\phi(Z)$ is invertible: then $C=\phi(X) E \phi(Z)^{-1}$. The subalgebra generated by $X,z$ and $y$ is isomorphic to $\mathbb{C}_\rho[x,y,z]$ and $\phi$ determines a simple representation with $\phi(z)=0$, therefore $E = f_1 e_2^{-1}e_1^{-1}$. This implies that $E$ only depends on 1 parameter and $C$ is uniquely determined by $E$.
\item $\phi(Z) = 0$: because $\phi(X)$ is invertible, $E = 0$. From the relations $Xx = xX$ and $xy = \rho yx$ one deduces that $C$ belongs to the vector space generated by $e_1$.
\end{itemize}
In both cases, we find that $D$ and $C$ depend on one parameter, from which smoothness follows.
\end{proof}

In fact, we find a similar decomposition for the normal as in proposition \ref{prop:normaldecomposition}.
\begin{proposition}
The normal $N_{\wis{PGL}_n}$ in a simple module $M$ of $B$ determined by matrices in standard form $(\phi(X),\phi(Z),\phi(x),\phi(y),\phi(z))$ with $\phi(y) \neq 0$ is determined by the subspace 
$$ V=(a_1 e_1,b_1 e_2^{-1}e_1^{-1},c_1 e_1,d_1 e_2,f_1 e_2^{-1}e_1^{-1}) \in \wis{M}_n(\C)^{\oplus 5}$$
and one extra relation coming from the relation $Xz=xZ$ holding in $B$.
\end{proposition}
\begin{proof}
This follows directly from proposition \ref{prop:normaldecomposition} as the algebra generated by $X,y,z$, the algebra generated by $x,y,Z$ and the algebra generated by $X,y,Z$ are all isomorphic to $A$. As we are working in $\wis{trep}^{ss}_n~B$ and $\phi(y)\neq 0$, $M$ is a simple module of at least two of the three subalgebras. It follows that $N_{\wis{PGL}_n}$ is indeed a subspace of $V$. The relation $Xz=xZ$ is the only relation we haven't used and it follows that this defines a non-trivial subspace of $V$, as we know that $N$ is 4-dimensional.
\end{proof}
However, this doesn't necessarily mean that $\wis{trep}^{ss}_n~B/\wis{PGL}_n \times \C^*$ is smooth, as the stabilizer is not necessarily trivial. 
\begin{theorem}
In the partial desingularization of the center $\wis{max}~Z(A)$ by $\wis{trep}^{ss}_n~B//\wis{PGL}_n \times \C^*$, for every point $\mathfrak{m} \neq (x^n,y^n,z^n,xyz)$ on the line $\mathbf{V}(x^n,z^n)$, we have that every point on $\pi^{-1}(\mathfrak{m}) = \mathbb{P}^1$ is smooth except for the points $0$ and $\infty$, where the singularity type is given by $\C \times \C^2/\Z_n$, where the action of $\Z_n$ for $0$ is defined by $\begin{bmatrix}\rho^2 & 0 \\ 0 & \rho^{-1}\end{bmatrix}$ and for $\infty$ by $\begin{bmatrix}\rho^{-2} & 0 \\ 0 & \rho\end{bmatrix}$.
\end{theorem}
\begin{proof}
We will calculate $N_{\wis{PGL}_n\times \C}$ at all points lying over $\mathfrak{m}$. We can assume that up to basechange the representation $\phi$ is given by the following matrices (in standard form)
\begin{gather*}
X \mapsto a e_1, Z \mapsto b e_2^{-1}e_1^{-1}, x \mapsto 0, y \mapsto d e_2, z \mapsto 0, 
\end{gather*}
and suppose first that $a \neq 0 \neq b$ and by assumption $d \neq 0$. The $\wis{PGL}_n\times \C$-stabilizer of $e_2$ is equal to $T_n/\C^*I_n \times \C^*$ (as the $\C^*$ action is trivial because $y$ is in degree 0). In order to have a non-trivial stabilizer, we must have that an element $g \in T_n/\C^*$ acting on $\phi(X)$ is a multiple of $\phi(X)$ or equivalently, $g^{-1} e_1 g = \lambda e_1$. This forces that $g = e_2^k$ for some $k \in \N$. We may assume that $g = e_2$ and so the cyclic subgroup $\langle e_2,\rho^{-1} \rangle$ stabilizes $\phi(X)$. However, if we calculate the action of any element of this subgroup on $\phi(Z)$, we get
$$
\rho^{-k} e_2^{-k} b e_2^{-1} e_1^{-1} e_2^{k}= \rho^{-2k} b e_2^{-1} e_1^{-1}
$$
If $(2,n) = 1$, $\rho^{-2k} = 1$ if and only if $k$ is a multiple of $n$. So the stabilizer is trivial in this case and consequently, the tangent space at the corresponding point of $\wis{trep}^{ss}_n~B//\wis{PGL}_n \times \C^*$ is 3-dimensional. If $2$ divides $n$, say $2n=k$, then there is indeed a stabilizer isomorphic to $\Z_2$ determined by $\langle (e_2^k,\rho^{-k}) \rangle$. If we look at the action of this stabilizer on $N_{\wis{PGL}_n}$, we find 
\begin{gather*}
(e_2^k,\rho^{-k}) \cdot e_1 = e_1 \\
(e_2^k,\rho^{-k}) \cdot e_2^{-1}e_1^{-1} =e_2^{-1}e_1^{-1}
\end{gather*}
for the degree 1 part and for the degree 0 part
\begin{gather*}
(e_2^k,\rho^{-k}) \cdot e_1 = -e_1 \\
(e_2^k,\rho^{-k}) \cdot e_2 = e_2\\
(e_2^k,\rho^{-k}) \cdot e_2^{-1}e_1^{-1} =-e_2^{-1}e_1^{-1}
\end{gather*}
We still need to divide out 1 relation determined by $Xz = xZ$, which amounts to removing one arrow of weight 1. This means that locally, the corresponding point in $\wis{trep}^{ss}_n~B//\wis{PGL}_n \times \C^*$ looks like $\C^3/\Z_2 = \C/\Z_2 \times \C^2$. As a consequence, $\wis{trep}^{ss}_n~B//\wis{PGL}_n \times \C^*$ is smooth in the corresponding point.

Suppose now that $b = 0$. The relation $xZ = Xz$ implies over the dual numbers that
$$
0 = \varepsilon a e_1 f_1 e_2^{-1}e_1^{-1}
$$
so $f_1 = 0$. In degree 1 we find
\begin{gather*}
(e_2,\rho^{-1})\cdot e_1 = e_1 \\
(e_2,\rho^{-1})\cdot e_2^{-1}e_1^{-1} = \rho^{-2} e_2^{-1}e_1^{-1}
\end{gather*}
and in degree 0
\begin{gather*}
(e_2,\rho^{-1}) \cdot e_1 = \rho e_1 \\
(e_2,\rho^{-1}) \cdot e_2 = e_2
\end{gather*}
This means that the weighted quiver setting associated to $N_{\wis{PGL}_n}$ is given by
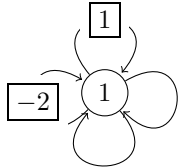
\begin{center}
\begin{tikzpicture}
   \node[vertice] (a) at (0,0) {$1$};
   \tikzset{every node/.style={fill=white}} 
   \path[->] (a) edge[out=30,in=-45,loop]  (a);
   \path[->] (a) edge[out=225,in=150,loop]node{$\boxed{-2}$} (a);
   \path[->] (a) edge[out=130,in=50,loop]node{$\boxed{1}$} (a);   
   \path[->] (a) edge[out=-50,in=-130,loop] (a);   
\end{tikzpicture}
\end{center}
To get $N_{\wis{PGL}_n\times \C^*}$, we need to divide out the action of $I_n+\varepsilon t$, which means taking away one arrow of weight 0
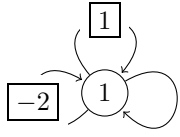
\begin{center}
\begin{tikzpicture}
   \node[vertice] (a) at (0,0) {$1$};
   \tikzset{every node/.style={fill=white}} 
   \path[->] (a) edge[out=30,in=-45,loop]  (a);
   \path[->] (a) edge[out=225,in=150,loop]node{$\boxed{-2}$} (a);
   \path[->] (a) edge[out=130,in=50,loop]node{$\boxed{1}$} (a);   
\end{tikzpicture}
\end{center}
A similar calculation shows that when $a = 0$, the stabilizer is given by $\langle (e_2,\rho) \rangle$ and in this case $c_1 = 0$. Similarly, one shows that $N_{\wis{PGL}_n\times \C^*}$ has weighted quiver setting 
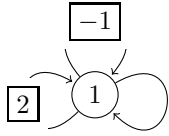
\begin{center}
\begin{tikzpicture}
   \node[vertice] (a) at (0,0) {$1$};
   \tikzset{every node/.style={fill=white}} 
   \path[->] (a) edge[out=30,in=-45,loop]  (a);
   \path[->] (a) edge[out=225,in=150,loop]node{$\boxed{2}$} (a);
   \path[->] (a) edge[out=130,in=50,loop]node{$\boxed{-1}$} (a);   
\end{tikzpicture}
\end{center}
leading to the claimed singularity type.
\end{proof}
The unadorned loop in the local quiver settings of the central singularities obtained in the theorem mean that there is a 1-dimensional family of similar singularities near each point.

Suppose now that $n$ is not divisible by 2. As $B$ is a finite module over its center $Z(B)$, it defines a coherent sheaf of Cayley-Hamilton algebras $\mathcal{B}$ over $\wis{proj}~Z(B)$ in the same way as that $A$ defined a sheaf $\mathcal{A}$ on $\PP^2$. On the open subset $X^n \neq 0$, we have
$$
\Gamma(\mathbb{X}(X^n),\mathcal{B})=\C \langle x,y,z,zx^{-1} \rangle \cong \frac{\C \langle u,v,w\rangle}{(uv-\rho vu,vw-\rho^2 wv,wu-\rho uw)}
$$
for which the representations are defined by $(a,b,c) \in \C^3$ and matrices
$$
u \mapsto a e_1, v \mapsto b e_2, w \mapsto c e_2^{-1}e_1^{-2}
$$
From this we easily see that above the point $(0,b^n,0,0) \in \wis{max}~Z(A)$ with $b \neq 0$ the following happens in $\wis{trep}_n~ \mathcal{B}$:
\begin{itemize}
\item if $c \neq 0$, then we know that the corresponding point of $\wis{proj}~Z(B)$ is smooth and as $(2,n) = 1$, $\mathcal{B}$ is Azumaya in this point. Therefore, $\mathcal{B}$ is Cayley smooth in this point of $\wis{proj}~Z(B)$.
\item if $c=0$, then we know that there is a singularity in the corresponding point of $\wis{trep}^{ss}_n~B//\wis{PGL}_n \times \C^*$. This representation is again a semi-simple representation that decomposes as a direct sum of 1-dimensional representations.
\end{itemize}
An analogous result as \ref{prop:normaldecomposition} holds for the normal $N_{\wis{PGL}_n}$ and we can prove
\begin{theorem}
The local quiver setting for a semi-simple but not simple representation in $\wis{trep}_n~\mathcal{B}$ over the point $(0,b^n,0,0) \in \wis{max}~Z(A), b \neq 0$ of $\mathcal{B}$ is given by 
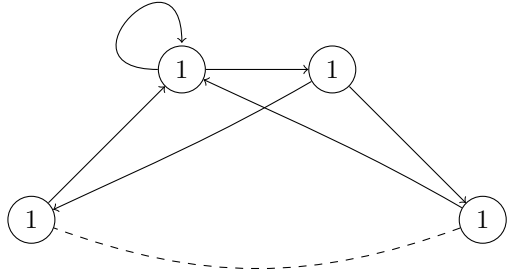
\begin{center}
\begin{tikzpicture}
   \node[vertice] (a) at (0,0) {$1$};
   \node[vertice] (b) at (2,0) {$1$};
   \node[vertice] (c) at (4,-2) {$1$};
   \node[vertice] (d) at (-2,-2) {$1$};
   
   \path[->] (a) edge[out=0,in=180] (b);
   \path[->] (b) edge[out=-150,in=25] (d);
   \path[->] (b) edge[out=-45,in=135] (c);
   \path[->] (c) edge[out=150,in=-25] (a);
     
   \path[->] (d) edge[out=45,in=-135] (a);   
   
   \path[-,dashed] (c) edge[out=-160,in=-20] (d);
   
   \path[->] (a) edge[out=180,in=90,loop] (a);
   
\end{tikzpicture}
\end{center}
Therefore, the defect stays equal to $n-1$.
\end{theorem}
Of course, the same is true for the blow-up algebra at the ideals $(x,y)$ and $(y,z)$.

\subsection{Blow-up of the origin} The next thing we want to do is to do a blow-up at the maximal ideal $\mathfrak{m}=(x,y,z)$. Let
$$
B = A \oplus\mathfrak{m}t \oplus \mathfrak{m}^2t^2 \oplus \ldots
$$
The dimension of $\wis{trep}^{ss}_n~B$ is again $n^2+3$. $B$ is again a Cayley-Hamilton algebra and a finite module over its center $Z(B))$. We want to study how $\wis{trep}^{ss}_n~B//\wis{PGL}_n \times \C^*$ looks like over the point $(0,0,0,0) \in \wis{max}~Z(A)$. Let $X=xt,Y=yt,Z=zt$. We find
\begin{theorem}
In the partial resolution of singularities determined by $\xymatrix{\wis{trep}^{ss}_n~B//\wis{PGL}_n \times \C^* \ar@{->>}[r]& \wis{max}~Z(A)}$, all the points over $\mathfrak{m}$ are smooth except for the three lines which form the point modules of $A$. 
\end{theorem}
\begin{proof}
We have that $\wis{trep}^{ss}_n~B$ is smooth above $\mathfrak{m}$ except for the 3 points  corresponding to the direct sum of $n$ 1-dimensional representations by Theorem 2 of \cite{DeLaetLeBruyn}. However, for Azumaya points with non-trivial $\wis{PGL}_n \times \C^*$-stabilizer (corresponding to points on the cone over the 3 lines except for the singular points), the weighted local quiver setting in $Z(B)$ is given by 
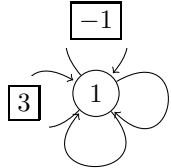
\begin{center}
\begin{tikzpicture}
   \node[vertice] (a) at (0,0) {$1$};
   \tikzset{every node/.style={fill=white}} 
   \path[->] (a) edge[out=30,in=-45,loop]  (a);
   \path[->] (a) edge[out=225,in=150,loop]node{$\boxed{3}$} (a);
   \path[->] (a) edge[out=130,in=50,loop]node{$\boxed{-1}$} (a); 
   \path[->] (a) edge[out=-50,in=-130,loop] (a);  
\end{tikzpicture}
\end{center}
as the $-1$-weight follows from the action of $\Z_n$ on the degree 0 variables.
Therefore, $\wis{proj}~Z(B)$ is not smooth on the corresponding lines (as the singular locus is closed), with everywhere except for the 3 singular points of the lines the singularity type given by $\C \times \C^2/\Z_n$.

For the points in $\wis{trep}^{ss}_n~B$ lying over $\mathfrak{m}$ that correspond to semi-simple but not simple representations, for example $x,y,z \mapsto 0, X,Z \mapsto 0, Y \mapsto b e_2$, $N_{\wis{PGL}_n}$ can be computed to be given by (using the fact that $Tr(yY^i)=0$ for $0\leq i \leq n-2$)
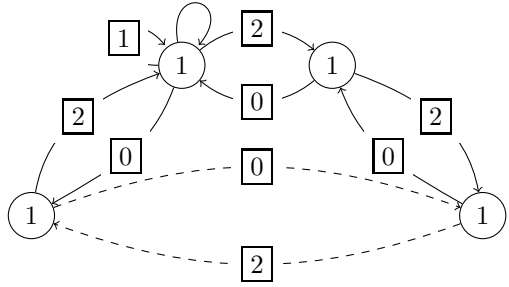
\begin{center}
\begin{tikzpicture}
   \node[vertice] (a) at (0,0) {$1$};
   \node[vertice] (b) at (2,0) {$1$};
   \node[vertice] (c) at (4,-2) {$1$};
   \node[vertice] (d) at (-2,-2) {$1$};
   
\tikzset{every node/.style={fill=white}}   
   
   \path[->] (a) edge[out=40,in=140]node{$\boxed{2}$} (b);
   \path[->] (b) edge[out=-140,in=-40]node{$\boxed{0}$} (a);
   \path[->] (b) edge[out=-20,in=100]node{$\boxed{2}$} (c);
   \path[->] (c) edge[out=150,in=-70]node{$\boxed{0}$} (b);
     
   \path[->] (a) edge[out=-110,in=30]node{$\boxed{0}$} (d);
   \path[->] (d) edge[out=80,in=-160]node{$\boxed{2}$} (a);   
   
   \path[->,dashed] (d) edge[out=20,in=160]node{$\boxed{0}$} (c);
   \path[->,dashed] (c) edge[out=-160,in=-20]node{$\boxed{2}$} (d);
   
   \path[->] (a) edge[out=180,in=130,loop]node{$\boxed{1}$} (a);
   \path[->] (a) edge[out=100,in=50,loop] (a);
   
\end{tikzpicture}
\end{center}
In order to get $N_{\wis{PGL}_n\times \C^*}$, we need to delete the loop of weight 1 (as the only non-trivial $\C^*$-action works on $Y$, which corresponds to the loop of weight 1).

\end{proof}
If we again work with $\mathcal{B}$, the coherent sheaf of Cayley-Hamilton algebras defined by $B$ over $\wis{proj}~Z(B)$, we again need to look at the global sections $\Gamma(\mathbb{X}(X^n),\mathcal{B})$, $\Gamma(\mathbb{X}(Y^n),\mathcal{B})$ and $\Gamma(\mathbb{X}(Z^n),\mathcal{B})$. We find that
$$
\Gamma(\mathbb{X}(X^n),\mathcal{B})\cong \Gamma(\mathbb{X}(Y^n),\mathcal{B})\cong \Gamma(\mathbb{X}(Z^n),\mathcal{B})\cong\frac{\C \langle u,v,w\rangle}{(uv-\rho vu,vw-\rho^3 wv,wu-\rho uw)}
$$
from which it follows that
\begin{theorem}
$\mathcal{B}$ is Azumaya away from the three lines in $\wis{proj}~Z(B)$ lying over $\mathfrak{m}$ and therefore is Cayley-smooth in these points. However, $\mathcal{B}$ is not Cayley-smooth over the 3 lines, with local quiver over a regular point of the 3 lines given by
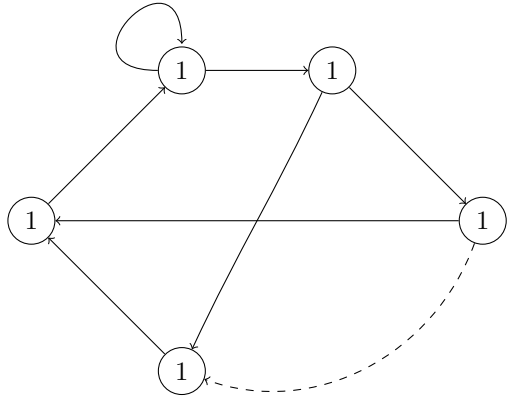
\begin{center}
\begin{tikzpicture}
   \node[vertice] (a) at (0,0) {$1$};
   \node[vertice] (b) at (2,0) {$1$};
   \node[vertice] (c) at (4,-2) {$1$};
   \node[vertice] (d) at (-2,-2) {$1$};
   \node[vertice] (e) at (0,-4) {$1$};
   
   \path[->] (a) edge[out=0,in=180] (b);
   \path[->] (b) edge[out=-45,in=135] (c);
     
   \path[->] (d) edge[out=45,in=-135] (a);   
   \path[->] (b) edge[out=-115,in=65] (e);   
   \path[->] (e) edge[out=135,in=-45] (d);   
   \path[->] (c) edge[out=180,in=0] (d);   
   
   \path[->,dashed] (c) edge[out=-110,in=-20] (e);
   
   \path[->] (a) edge[out=180,in=90,loop] (a);
   
\end{tikzpicture}
\end{center}
which is the Mckay quiver of $\C^2/\Z_n$ with $Z_n$ acting as the matrix $\begin{bmatrix}
\rho^3 & 0 \\ 0 & \rho^{-1}
\end{bmatrix}$ with one extra loop corresponding to the 1-dimensional family of similar points. 

In the singular points of the 3 lines, the local quiver is determined by 
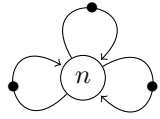
\begin{center}
\begin{tikzpicture}
   \node[vertice] (a) at (0,0) {$n$};
   
   \path[->] (a) edge[out=30,in=-45,loop] node {$\bullet$} (a);
   \path[->] (a) edge[out=120,in=45,loop] node {$\bullet$} (a);
   \path[->] (a) edge[out=225,in=150,loop] node {$\bullet$} (a);   
   
\end{tikzpicture}
\end{center}
\end{theorem}
Therefore, we see that, while the central singularities have become better, the singularities in $\wis{trep}_n~\mathcal{B}$ have become worse (we now have 3 points with the same type of singularity as the original one).

\end{document}